\documentclass{amsart}

\usepackage{amsmath, amsthm, amssymb, amstext, amsfonts}
\usepackage{enumerate}
\usepackage{graphicx}
\usepackage[T1]{fontenc} %needed for polic ogonek
\usepackage[applemac]{inputenc}

\theoremstyle{plain}

\newtheorem{thm}{Theorem}[section]

\newtheorem{cor}[thm]{Corollary}
\newtheorem{prop}[thm]{Proposition}

\newtheorem*{prob*}{Problem}

\theoremstyle{definition}

\newcommand{\N}{\ensuremath{\mathbb{N}}}

\newcommand{\cP}{\ensuremath{\mathcal{P}}}
\newcommand{\cQ}{\ensuremath{\mathcal{Q}}}

\newcommand{\inv}{\ensuremath{^{-1}}}

\newcommand{\es}{\ensuremath{\emptyset}}

\newcommand{\sub}{\subseteq}

\def\qt{quasi-tran\-si\-tive}
\def\qi{quasi-iso\-metric}
\def\qiy{quasi-iso\-me\-try}

\newcommand{\comment}[1]{}

\newenvironment{txteq*}
  {
    \begin{equation*}
    \begin{minipage}[c]{0.85\textwidth} % set width to 0.9 x textwidth
    \em                                % switch on emph
  }
  {\end{minipage}\end{equation*}\ignorespacesafterend}

\begin{document}

\title{Minor exclusion in quasi-transitive graphs}
\author{Matthias Hamann}
\thanks{Funded by the German Research Foundation (DFG) -- project number 448831303.}
\address{Matthias Hamann, Mathematics Institute, University of Warwick, Coventry, UK}

\date{}

\begin{abstract}
In this note, we show that locally finite \qt\ graphs are \qi\ to trees if and only if every other locally finite \qt\ graph \qi\ to them is minor excluded.
This generalizes results by Ostrovskii and Rosenthal and by Khukhro on minor exclusion for groups.
\end{abstract}

\maketitle

\section{Introduction}

For two graphs $G,H$, we call $H$ a \emph{minor} of~$G$ if $H$ can be obtained from~$G$ by contracting edges and deleting edges and vertices.
A graph is \emph{minor excluded} if there exists some finite graph that is not a minor of it.

In graphs, minor exclusion has played an important role for a long time, e.\,g.\ via Kuratowski's planarity criterion.
Considering groups, minor exclusion was mostly considered in the case of planar groups, where a finitely generated group is \emph{planar} if it has a planar locally finite Cayley graph.
Ostrovskii and Rosenthal~\cite{OR-MinorExclusionGroups} looked at minor exclusion for groups from a broader viewpoint: do there exists locally finite groups all of whose locally finite Cayley graphs are minor excluded (not minor excluded)?
They answered both questions positively:
they proved for an infinite class of finitely generated groups that all of their locally finite Cayley graphs are not minor excluded and they proved that every locally finite Cayley graph of any finitely generated virtually free group is minor excluded.
This latter result was extended by Khukhro \cite{K-MinorExclusionFreeGroups}.
She showed the reverse direction, i.\,e.\ characterised the finitely generated groups all of whose locally finite Cayley graphs are minor excluded as the finitely generated virtually free groups.

We generalise this characterisation to \qt\ graphs, where a graph is \emph{\qt} if its automorphism group acts on it with only finitely many orbits.
The analogue of looking at Cayley graphs in this situation is that we ask for minor exclusion for all locally finite \qt\ graphs that are \qi\ to the original one.
For two graph $G$ and~$H$ a map $\varphi\colon V(G)\to V(H)$ is \emph{\qi} if there exist $\gamma\geq 1$ and $c\geq 0$ such that the following holds for all $x,y\in (G)$:
\[
\frac{1}{\gamma}d_H(\varphi(x),\varphi(y))-c\leq d_G(x,y)\leq\gamma d_H(\varphi(x),\varphi(y))+c,
\]
where $d_G$ and $d_H$ denote the distance functions in~$G$ and~$H$, respectively.
We will prove the following theorem.

\begin{thm}\label{thm_Main}
Let $G$ be a locally finite \qt\ graph.
Then $G$ is \qi\ to a tree, if and only if every locally finite \qt\ graph \qi\ to~$G$ is minor excluded.
\end{thm}

Bonamy et al.~\cite{BBEGPS-SurfaceAsDim2} proved that all locally finite Cayley graphs of finitely generated groups of asymptotic dimension at least~$3$ are not minor excluded, moreover from their discussion follows that all \qt\ locally finite graphs of asymptotic dimension at least~$3$ are not minor excluded.
On the other side, \qt\ locally finite graph that are \qi\ to trees have asymptotic dimension~$1$.
However there are finitely generated groups of asymptotic dimension~$1$ that are not virtually free, see Gentimis~\cite{G-AsdimFinPresGroup}.
The following question remains open: do there exist locally finite \qt\ graphs of asymptotic dimension at most~$2$ such that all locally finite \qt\ graphs \qi\ to it are not minor excluded?

\section{Proof}

Before we start the proofs, we need some definitions.
Let $G$ and~$H$ be graphs.
We call $H$ a \emph{minor} of~$G$, if there exists a set $\{G_x\mid x\in V(H)\}$ of disjoint subsets of~$V(G)$ such that if $xy\in E(H)$, then there exists $uv\in E(G)$ with $u\in G_x$ and $v\in G_y$.
The vertex sets $G_x$ are the \emph{branch sets}.

By $G^2$ we denote the graph with vertex set $V(G)$ such that two vertices are adjacent if and only if their distance in~$G$ is either $1$ or~$2$.

A \emph{ray} is a one-way infinite path.
Two rays in~$G$ are \emph{equivalent} if for every finite $S\sub V(G)$ there exists a component of $G-S$ that contains all but finitely many vertices of both rays.
This is an equivalence relation whose classes are the \emph{ends} of~$G$.
Let $m\in\N$.
An end has \emph{degree at least $m$} if it contains $m$ pairwise disjoint rays.
An end is \emph{thin} if there exists an $n\in\N$ such that the end does not have degree at least~$n$.
It is \emph{thick} if it is not thin.

\begin{prop}\label{prop_main}
Let $G$ be an infinite graph with an end of degree at least~$m$.
Then $G^2$ contains a $K_m$-minor.
\end{prop}

\begin{proof}
Let $\omega$ be an end of~$G$ of degree at least~$m$ and let $R_1,\ldots, R_m$ be $m$ disjoint rays in~$\omega$.
Let $\cP_1,\ldots,\cP_n$ with $n=\frac{m(m-1)}{2}$ be an enumeration of the two-element subsets of $\{ R_1,\ldots, R_m\}$.

For all $i\in\{0,\ldots,n\}$ we construct a finite vertex set $S_i$ of~$G$, $m$ disjoint rays $R_1^i,\ldots,R_m^i$ in~$R^2$ and a set $\cQ_i$ of~$i$ disjoint paths in~$R^2$ such that the following holds:
\begin{enumerate}[(i)]
\item $R_m^i$ and $R_m$ coincide outside of~$S_i$;
\item the last vertex of~$R_m^i$ in~$S_i$ lies on~$R_m$;
\item the vertices of~$R_m^i$ outside of~$S_i$ form a tail of~$R_m^i$;
\item every $Q\in\cQ_i$ is internally disjoint from all $R_j^i$;
\item every $Q\in \cQ_i$ has all its vertices in~$S_i$;
\item for every $1\leq j\leq i$, there is a path $Q_j\in \cQ_i$ joining the rays in $\cP_j$.
\end{enumerate}
Note that (i) implies that $R_m^i$ has a tail in~$G$.

For $i=0$, set $S_i=\es$ and $\cQ_i=\es$ and $R_j^i=R_j$ for all $1\leq j\leq m$.
This satisfies (i)--(vi) trivially.

Now let us assume that we have constructed $S_{i-1}$, $R_1^{i-1},\ldots, R_m^{i-1}$ and $\cQ_{i-1}$.
Let $\cP_i=\{R_k,R_\ell\}$ and let $Q$ be a path in~$G-S_{i-1}$ joining $R_k$ and $R_\ell$.
Note that $Q$ also joins $R_k^{i-1}$ and $R_\ell^{i-1}$.
We may assume that $Q$ meets those two rays only in its end vertices.
Let $R_{i_1}^{i-1}$ be the ray that $Q$ meets first after its starting vertex.
Let $x_1$ be the first common vertex of $Q$ and $R_{i_1}^{i-1}$ and let $x_2$ be the last such vertex.
We replace $x_1Qx_2$ by a path in $G^2$ consisting of all vertices on $x_1R_{i_1}^{i-1}x_2$ with even distance on that ray to~$x_1$ and we replace $x_1R_{i_1}^{i-1}x_2$ by the path in~$G^2$ of all vertices with odd distance to~$x_1$ on that ray.
The resulting path does not meet the new ray $R_{i_1}^i$ at all.
We continue doing these modifications for all remaining intersections of the new path with other rays $R_{i_j}^{i-1}$, where we put the vertices with even distance to the first vertex of~$Q$ and $R_{i_j}^{i-1}$ into the modification of~$Q$ if the last common vertex of~$R_{i_{j-1}}^{i-1}$ and~$Q$ was not put into that path and otherwise the vertices with odd distance.
If we have not modified the ray $R_j^{i-1}$, then we set $R_j^i:=R_j^{i-1}$.
By the choice of when to put the first common vertex of~$Q$ and the rays into the modification of~$Q$ or the new rays, the resulting path~$Q'$ and the rays $R_{i_j}^i$ are indeed a path and rays.
Then $Q'$ will be added to the set $\cQ_{i-1}$ to obtain the set~$\cQ_i$.
Let $S_i$ be $S_{i-1}$ together with all vertices of~$Q'$ and, for each $1\leq j\leq m$, a finite starting path of~$R_j^i$ such that $R_j^i$ coincides with $R_j$ after this path and such that the last vertex of that path also lies on~$R_j$.
By construction, (i)--(vi) hold.

For $1\leq i\leq m$, let $P_i$ be a subpath of~$R_i^n$ that contains all vertices of~$R_i^n$ that lie on paths of~$\cQ_n$.
Then it is easy to see that $G^2$ contains a $K_m$-minor, where each $P_i$ lies in a different branch set and each $Q\in\cQ_n$ that connects $P_i$ and $P_j$ is split among those branch sets.
\end{proof}

Using a result of Thomassen~\cite{T-Hadwiger}, we obtain the following corollary of Proposition~\ref{prop_main}.

\begin{cor}
Let $G$ be a one-ended \qt\ locally finite graph.
Then $G^2$ is not minor excluded.
\end{cor}

\begin{proof}
By Thomassen \cite[Proposition 5.6]{T-Hadwiger}, the unique end of a one-ended \qt\ locally finite graph is thick.
Thus, Proposition~\ref{prop_main} implies the assertion.
\end{proof}

\begin{prop}\label{prop_QITreeIsMinEx}
Let $G$ be a graph of bounded degree that is \qi\ to a tree of bounded degree.
Then $G$ is minor excluded.
\end{prop}

\begin{proof}
Let $T$ be a tree of bounded degree and let $\varphi\colon V(G)\to V(T)$ be a $(\gamma,c)$-\qiy\ for some $\gamma\geq 1$ and $c\geq 0$.
Let $D_G$, $D_T$ be the maximum degrees of~$G$ and~$T$, respectively.
Then
\[
M_T:=\sum_{i=0}^{\gamma+c-2}(D_T-1)^i,\qquad M_G:=\sum_{i=0}^{c-1}(D_G-1)^i
\]
are the maximum sizes of balls around vertices of~$T$, of~$G$ of radius $\gamma+c-1$, of radius $c$, respectively.
Let $H$ be a minor of~$G$ that is isomorphic to~$K_n$ for some $n\in\N$.
We will show that $n\leq\max\{2M_T^2 M_G^2,M_G(D_TM_T+1)\}$, which implies that $G$ is minor excluded.
For $x\in V(H)$, we denote by $G_x$ the branch set of~$x$ in~$G$.
For an edge $e=st\in E(T)$, let $T_s$, $T_t$ be the component of $T-e$ that contains~$s$, $t$, respectively.
Note that since $\varphi$ is a $(\gamma,c)$-\qiy, we have $|\varphi\inv(t)|\leq M_G$.

Let us assume that there is an edge $e=st\in E(T)$ and $x,y\in V(H)$ with $\varphi(G_x)\sub V(T_s)$ and $\varphi(G_y)\sub V(T_t)$.
Since $H$ is a complete graph, there is a $G_x$-$G_y$ edge $uv$ in~$G$.
By the assumption on~$\varphi$, we have $d(\varphi(u),\varphi(v))\leq \gamma+c$.
Let us assume that $\varphi(u)\in V(T_s)$.
Then $\varphi(u)$ lies in the ball of radius at most $\gamma+c-1$ around~$s$ and $\varphi(v)$ lies in the ball of radius at most $\gamma+c-1$ around~$t$.
There are at most $M_T^2$ many such pairs $(u,v)$.
Since $|\varphi\inv(a)|\leq M_G$ for every $a\in V(T)$, there are at most $M_T^2 M_G^2$ many such pairs $(u,v)$.
Now let $z\in V(H)$ such that $\varphi(G_z)$ meets $T_s$ and~$T_t$.
Then there are adjacent $u,v\in G_z$ with $\varphi(u)\in V(T_s)$ and $\varphi(v)\in V(T_t)$.
So as above, we find at most $M_T^2 M_G^2$ such edges $uv$ and hence such vertices~$z$.
This implies $n\leq 2M_T^2 M_G^2$.

Now let us assume that for every edge $e=st\in E(T)$ there is at most one component of $T-e$ that contains some $\varphi(G_x)$.
If $T_s$, $T_t$ contains some $\varphi(G_x)$, then we orient the edge $e$ towards $s$, $t$, respectively.
If neither $T_s$ nor $T_t$ contains any $\varphi(G_x)$, we do not orient $e$ at all.
Note that this orientation is \emph{consistent}: if $s_1t_1, s_2t_2\in E(T)$ and they are oriented towards $s_1$, $s_2$, respectively, then we must have either $T_{s_2}\sub T_{s_1}$ or $T_{s_1}\sub T_{s_2}$.
Thus and since $H$ is finite, there is a unique non-empty subtree $T'$ of~$T$ whose inner edges are not directed at all but such that every edge outside of~$T'$ is directed towards~$T'$.
If $T'$ has an edge $e=st$, then for every $x\in V(H)$, the set $\varphi(G_x)$ meets $T_s$ and~$T_t$.
By the same argument as in the previous case, there are at most $M_T^2M_G^2$ such branch sets.
Thus, we have $n\leq M_T^2M_G^2$ in this case.
So let us assume that $T'$ consists of a unique vertex~$t$.
Then every branch set $G_x$ contains a vertex $u$ with $d(\varphi(u),t)\leq \gamma+c$.
There are at most $D_TM_T+1$ many possibilities for~$\varphi(u)$ and thus at most $M_G(D_TM_T+1)$ many possibilities for~$u$.
This implies $n\leq M_G(D_TM_T+1)$ in this situation.

So we have $n\leq \max\{2M_T^2 M_G^2,M_G(D_TM_T+1)\}$, which proves the assertion as discussed above.
\end{proof}

Now we are able to prove our main result.

\begin{proof}[Proof of Theorem~\ref{thm_Main}.]
Assume that $G$ is not \qi\ to any tree.
Since locally finite \qt\ graphs without thick end are \qi\ to trees by Kr\"on and M\"oller \cite[Theorem 2.8]{KM-QuasiIsometries}, $G$ has a thick end.
Thus, $G^2$ is not minor excluded by Proposition~\ref{prop_main}.
Since $G^2$ is \qi\ to~$G$, it is not minor excluded, either.

Let us now assume that $G$ is \qi\ to a tree.
Note that $G$ has bounded degree by assumption.
According to Kr\"on and M\"oller~\cite[Theorem 2.8]{KM-QuasiIsometries}, every end of~$G$ is thin.
Thus, \cite[Theorem 7.5]{HLMR} and \cite[Lemma 2.9]{H-QuasiIsometriesAndTA} imply that $G$ is \qi\ to a 3-regular tree.
By Proposition~\ref{prop_QITreeIsMinEx}, $G$ is minor excluded.
Since every graph that is \qi\ to~$G$ is also \qi\ to a tree, they are minor excluded as well.
\end{proof}

\end{document}